\documentclass[10pt]{amsart}
\usepackage{graphicx,amssymb,amsfonts,amsmath,amsthm,newlfont}
\usepackage{epsfig,url}

\usepackage[all,2cell]{xy} \UseAllTwocells \SilentMatrices

\newtheorem{thm}{Theorem}[section]
\newtheorem{cor}[thm]{Corollary}
\newtheorem{lem}[thm]{Lemma}

\theoremstyle{definition}
\newtheorem{defn}[thm]{Definition}
\newtheorem{conj}{Conjecture}

\theoremstyle{remark}
\newtheorem{rem}[thm]{Remark}
\numberwithin{equation}{section}

\newcommand{\Z}{\mathbb Z}

\newcommand{\R}{\mathbb R}
\newcommand{\N}{\mathbb N}
\newcommand{\Pro}{\mathbb P}

\newcommand{\gr}{\mathrm{gr}}

\font \rus= wncyr10
\newcommand{\sha}{\, \hbox{\rus x} \,}

\newcommand{\Ho}{\mathcal{H}}
\newcommand{\Homt}{\mathcal{H^{MT_+}}}

\newcommand{\MT}{\mathcal{MT}}
\newcommand{\GMT}{\mathcal{G_{MT}}}
\newcommand{\G}{\mathbb{G}}
\newcommand{\GU}{\mathcal{G}_{\mathcal{U}}}
\newcommand{\GP}{\mathcal{G_{MT'}}}
\newcommand{\GPU}{\mathcal{G_{\mathcal{U'}}}}
\newcommand{\Ao}{\mathcal{A}}
\newcommand{\tp}{{}_0 1_{1}}
\newcommand{\pio}{{}_0 \Pi_{1}}
\newcommand{\dch}{dch}
\newcommand{\Jmt}{J^{\MT} }
\newcommand{\ZZ}{\mathcal{Y}}

\newcommand{\Io}{\mathbf{I0}}
\newcommand{\Ii}{\mathbf{I1}}
\newcommand{\Iii}{\mathbf{I2}}
\newcommand{\Iiii}{\mathbf{I3}}

\newcommand{\zetam}{\zeta^{ \mathfrak{m}}}

\newcommand{\Imot}{I^{\mathfrak{m}}}

\newcommand{\Q}{\mathbb Q}

\newcommand{\Lo}{\mathcal{L}}
\newcommand{\U}{\mathcal{U}}

\newcommand{\To}{\longrightarrow}

\newcommand{\A}{\mathbb{A}}
\newcommand{\Amt}{\mathcal{A}^{\MT}}

\newcommand{\tdots}{.\, .\,}
\newcommand{\oo}{\mathbf{00}}
\newcommand{\os}{\mathbf{0}}

\def\0skip{\vskip 2pt}
\def\1skip{\vskip 4pt}
\def\2skip{\vskip 10pt}
\def\3skip{\vskip 0.15in}
\def\4skip{\vskip 0.20in}
\def\5skip{\vskip 0.25in}
\def\6skip{\vskip 0.30in}

\begin{document}
\author{Francis Brown}
\begin{title}[Mixed Tate motives over $\Z$]{Mixed Tate motives over $\Z$}\end{title}
\maketitle
\begin{abstract} We prove that the category of mixed Tate motives over $\Z$ is spanned by the motivic fundamental group of $\Pro^1$ minus three points.
We  prove a conjecture by M. Hoffman which states
 that every multiple zeta value is a $\Q$-linear combination of $\zeta(n_1,\ldots, n_r)$ where $n_i\in \{2,3\}$.
\end{abstract}
\section{Introduction}
Let $\MT(\Z)$ denote the category of mixed Tate motives unramified over $\Z$.
It is a Tannakian category with Galois group $\GMT$. Let
 $\MT'(\Z)$ denote the  full  Tannakian subcategory generated by   the  motivic fundamental group of $\Pro^1\backslash \{0,1,\infty\}$, and denote its
 Galois group by $\GP$.   The following conjecture is  well-known: 

\begin{conj} \label{conj1} The map $\GMT\rightarrow \!\!\!\!\!\rightarrow  \GP$ is an isomorphism.
\end{conj}
Some consequences of this  conjecture are   explained  in   \cite{Andre}, $\S25.5-7$.
In particular, it implies a conjecture due to Deligne and Ihara on the  outer action of $\mathrm{Gal}(\overline{\Q}/\Q)$ on the pro-$\ell$ fundamental group of $\Pro^1\backslash \{0,1,\infty\}$.  
Another  consequence is that the periods of $\MT(\Z)$ are $\Q[{1\over 2 \pi i }]$-linear combinations of 
multiple zeta values
\begin{equation}  \label{introzetadef}
\zeta(n_1,\ldots, n_r) = \sum_{0< k_1 < \ldots < k_r}  {1\over k_1^{n_1}\ldots k_r^{n_r}}  \quad \hbox{ where }  n_i\geq 1\ ,\  n_r\geq 2 \ .
\end{equation}

On the other hand, M. Hoffman proposed a  conjectural basis for the $\Q$-vector space spanned by multiple zeta values in \cite{Hoffman}.  The algebraic part of this conjecture is
\begin{conj} \label{conj2} Every multiple zeta value $(\ref{introzetadef})$ is a $\Q$-linear combination of 
\begin{equation}\label{Hbasis} \{\zeta(n_1,\ldots, n_r): \hbox{ where } n_1,\ldots, n_r \in \{2,3\}\} \ .
\end{equation}
\end{conj}
In this paper we prove conjectures \ref{conj1} and \ref{conj2} using  motivic multiple zeta values. These are  elements in 
a certain graded comodule $\Homt$ over the affine ring   of functions on the prounipotent part of $\GMT$,  and are graded versions of the motivic iterated integrals
defined in \cite{GG}. We  denote each motivic multiple zeta  value by a symbol 
\begin{equation}\label{introzmot}
\zetam(n_1,\ldots, n_r) \quad  \hbox{ where }   n_i \geq 1 \ , \ n_r \geq 2 \ ,
\end{equation}
and its period  is the  multiple zeta value $(\ref{introzetadef})$.
 Note that in our setting $\zetam(2)$ is not zero, by contrast with \cite{GG}.
Our main result is the following:
\begin{thm} \label{intromaintheorem}  The set of elements 
\begin{equation} \label{intromotHbasis} \{\zetam(n_1,\ldots, n_r)\ , \hbox{ where  } n_i\in \{2,3\}\}\ ,
\end{equation} 
 are a basis of the  $\Q$-vector space  of  motivic multiple zeta values.\end{thm}
Since the dimension of the basis $(\ref{intromotHbasis})$ coincides with the known dimension for  $\Homt$ in each  degree,
 this  yields  conjecture \ref{conj1}. Conjecture \ref{conj2} follows  from theorem \ref{intromaintheorem} by applying the period map. Both conjectures together imply the following

\begin{cor} The periods of every mixed Tate motive over $\Z$ are $\Q[{1\over 2 \pi i}]$-linear combinations of 
$\zeta(n_1,\ldots, n_r)$,  where $n_1,\ldots, n_r\in \{2,3\}$.
\end{cor}

\subsection{Outline} The structure of  the de Rham realization of  $\GMT$ is well-known: there is a split exact sequence 
$$1 \To \GU \To \GMT \To \G_m \To 1\ , $$
where $\GU$ is a prounipotent group whose   Lie algebra is free,  generated by  one element $\sigma_{2n+1}$ in  degree $-2n-1$, for all $n\geq 1$. 
Let   $\Amt$ 
denote its graded  affine ring over $\Q$. It  is a cofree commutative graded Hopf algebra cogenerated by one element  $f_{2n+1}$ in  degree $2n+1$, for  all $n\geq 1$.
Consider the free comodule over $\Amt$ defined by
$$\Homt= \Amt\otimes_{\Q} \Q[f_2]\ ,$$
 where $f_2$ is  in degree $2$, has trivial coaction,  and is  an artefact to keep track of  even Tate twists (since multiple zeta values are real numbers, we need not consider odd Tate twists).   In keeping with the usual terminology for multiple zeta values,  we refer to the grading on $\Homt$ as the weight, which is one half the motivic weight.
  The motivic multiple zeta values  $(\ref{introzmot})$ we shall use   are    elements of $\Homt$ defined  by     functions on a certain subscheme of  the motivic torsor of paths of $\Pro^1\backslash \{0,1,\infty\}$ from $0$ to $1$, and depend on some choices.
  They have a canonical period given by a  coefficient in Drinfeld's associator, and the element $\zetam(2)$, which is non-zero in our setting, corresponds to $f_2$. 
  
  Let $\Ho\subseteq \Homt$ denote the subspace   spanned by the motivic multiple zetas. By Ihara,
the action  of $\GU$  on the motivic torsor of paths is determined by its action on the trivial de Rham path from $0$ to $1$. The dual coaction
 \begin{equation} \label{introcoaction}
  \Delta:\Ho\To \Amt \otimes_{\Q} \Ho
  \end{equation}  can be  determined   by a  formula due to   Goncharov  \cite{GG}.
  
  Let $\Ho^{2,3}\subseteq \Ho$  be the  vector subspace  spanned by the elements $(\ref{intromotHbasis})$. We   define an increasing filtration $F_{\bullet} $ on $\Ho^{2,3}$, called the level, by the number of arguments $n_i$  in   $\zetam(n_1,\ldots, n_m)$  which are equal to $3$.  We show that $\Ho^{2,3}$ and the $F_{\ell} \Ho^{2,3}$ are stable under the action of $\GU$, and that $\GU$ acts trivially on the $\gr^F_{\ell}(\Ho^{2,3})$. As a consequence, the action of $\GU$ on $F_{\ell} \Ho^{2,3} / F_{\ell-2} \Ho^{2,3}$ factors through the abelianization $\GU^{ab}$ of $\GU$.
  By construction, $\gr^F_{\ell} \Ho^{2,3}_N$ is spanned by elements $(\ref{intromotHbasis})$  indexed by the words in the letters $2$ and $3$, with $\ell$ letters $3$, and $m$ letters $2$, where $3\ell + 2 m = N$.  Let $(\gr^F_{\ell} \Ho^{2,3})^{\sim}$ be the vector space generated by the same words.
   The commutative Lie algebra $\mathrm{Lie}\, \GU^{ab}$ is generated by one element in every degree $-2i-1$ ($i\geq 1$). We compute their actions:
 \begin{equation} \label{introDNL} \partial_{N,\ell} :\, \gr^F_\ell \Ho_N^{2,3} \To \bigoplus_{1 <2i+1\leq N}  \gr^F_{\ell-1}\,  \Ho_{N-2i-1}^{2,3} 
\end{equation}
by constructing maps 
 \begin{equation} \label{introDNLformal} \partial^{\sim}_{N,\ell} :\, (\gr^F_\ell \Ho_N^{2,3})^{\sim} \To \bigoplus_{1 <2i+1\leq N}  (\gr^F_{\ell-1}\,  \Ho_{N-2i-1}^{2,3} )^{\sim}\ ,
\end{equation}
such that the following diagram commutes:
$$\begin{array}{ccc}
  (\gr^F_\ell \Ho_N^{2,3})^{\sim}  &   \To  &   \bigoplus_{1 <2i+1\leq N}  (\gr^F_{\ell-1}\,  \Ho_{N-2i-1}^{2,3} )^{\sim} \\
 \downarrow &   &   \downarrow \\
 \gr^F_\ell \Ho_N^{2,3}   &   \To  & \bigoplus_{1 <2i+1\leq N}  \gr^F_{\ell-1}\,  \Ho_{N-2i-1}^{2,3}   \ . 
\end{array}
$$
Using the explicit formula for the coaction $(\ref{introcoaction})$, we write  the maps $\partial^{\sim}_{N,\ell}$ as  matrices $M_{N,\ell}$ whose entries are linear combinations of certain 
rational numbers
$c_w \in \Q,$ 
where $w$  is a  word in  $\{2,3\}$ which has a single  $3.$
The numbers $c_w$ are defined as follows. We prove that for all $a,b\in \N$,  there exist numbers $\alpha^{a,b}_i\in \Q$ such that
\begin{equation} \label{introZagmot}
\zetam(\underbrace{2,\ldots, 2}_{a}, 3,\underbrace{2,\ldots, 2}_{b})= \alpha_n^{a,b} \zetam(2n+1)+  \sum_{i=1}^{n-1} \alpha^{a,b}_i\, \zetam(2i+1)  \zetam(\underbrace{2,\ldots, 2}_{n-i})\,   \ .\end{equation}
where $n=a+b+1$.   For any word $w$   of the form $2^{\{a\}}32^{\{b\}}$,   the number $c_w$
is the coefficient $\alpha^{a,b}_{n}$ of    $\zetam(2n+1)$  in $(\ref{introZagmot})$.  At this point we use a crucial arithmetic result due to Don Zagier  \cite{Zagier}, 
who  proved an explicit formula for $\zeta(2,\ldots, 2, 3,2,\ldots, 2)$ 
in terms of $\zeta(2i+1)$ and powers of $\pi$ of the same shape as $(\ref{introZagmot})$. Since the transcendence conjectures for multiple zeta values are not known, this does not immediately imply a formula for the coefficients $\alpha^{a,b}_i$. However, 
in $\S\ref{sect4}$ we show how to lift Zagier's  theorem  from  real numbers  to  motivic multiple zeta values, which 
yields a  precise formula for the coefficients $\alpha^{a,b}_i$, and in particular, $c_w$.   From this, we deduce  that the $c_w$ satisfy many special $2$-adic properties. By exploiting these properties, we show that the matrices
 $M_{N,\ell}$, which are rather complicated,  are  upper-triangular to leading $2$-adic order. 
 From this, we show that the maps  $(\ref{introDNL})$ are invertible for $\ell \geq 1$, and theorem   
$\ref{intromaintheorem} $  follows by an induction on the level. The proof  also shows that the level filtration is dual to the filtration induced by  the descending central series of $\GU$.

 P. Deligne has  obtained analogous results in the case $\Pro^1\backslash \{0, \mu_N,\infty\}$, where $\mu_N$ is the set of $N^{\mathrm{th}}$ roots of unity, and $N=2,3,4,6$  or $8$ 
 \cite{D}. His argument  uses the depth filtration and proves that it is dual to the filtration induced by the descending central series of $\GU$. 
  Note that  in the case $N=1$  this is false, starting from weight 12.

  The notes  \cite{Br} might serve as an introduction to this paper.

\section{Motivic Multiple Zeta Values}

\subsection{Preliminaries} \label{sectprelim}
With  the notations of \cite{DG}, \S5.13, let 
$\pio$ 
denote the de Rham realization of the motivic torsor of paths on $\Pro^1\backslash \{0,1,\infty\}$ from $0$ to $1$ (with tangent vectors $1$, $-1$ respectively).
It is the functor which to any $\Q$-algebra $R$ associates the 
 set $\pio(R)$ of group-like series in  the algebra $R\langle \langle e_0,e_1\rangle\rangle$ of non-commutative formal
power series in two generators. Its ring of affine functions over $\Q$ is isomorphic to 
\begin{equation}\label{opi01} \mathcal{O}(\pio) \cong \Q \langle e^0, e^1 \rangle\ ,
\end{equation} 
which is a commutative, graded algebra equipped with the shuffle product. To every word $w$ in the letters  $e^0,e^1$ corresponds the function 
which maps a series $S \in \pio(R) $ to the coefficient of $w$ (viewed as a word in $e_0,e_1$) in $S$.

Let $\dch\in \pio(\R)$ denote the de Rham image of the straight line from $0$ to $1$ (\cite{DG}, \S5.16).  
 It is a formal power series whose  coefficients are multiple zeta values, and is   also known as the Drinfeld associator.  It defines a function 
 \begin{equation}\label{dchdual}
 \dch : \mathcal{O}(\pio)\To \R
  \end{equation}
which maps a word $w$  in $e^0, e^1$ to the coefficient of $w$ in $dch$. 

Since $\pio$ is the de Rham realization of an (ind-)  mixed Tate motive over $\Z$, the group $\GMT$ acts upon it (\cite{DG}, \S5.12). The group $\GP$  in the introduction is the 
quotient of $\GMT$ by the kernel of this action.    Let $\GPU$ denote the corresponding quotient of $\GU$.
We shall denote the graded ring of affine functions on  $\GPU$ over $\Q$ by
$$\Ao=\mathcal{O}(\GPU)\ .$$
   The action $\GPU\times \pio \rightarrow \pio$ of the prounipotent part of $\GP$  gives rise to a coaction
  \begin{equation} \label{Ocoaction}
  \mathcal{O}(\pio)\To \Ao \otimes_{\Q} \mathcal{O}(\pio) . 
  \end{equation}

  Now let $A$ denote the group of automorphisms of the de Rham fundamental groupoid of $\Pro^1 \backslash \{0,1,\infty\}$ with base points $0,1$ and which respects $e_0\in \mathrm{Lie}\, ({}_0 \Pi_{0})$ and  $e_1\in \mathrm{Lie}\, ({}_1 \Pi_{1})$.  The action of $\GMT$ on $\pio$ factors through the action of  $A$ on $\pio$. The latter was computed by Ihara as follows (see \cite{DG}, \S5.15). Let   $ \tp$ denote the identity element in  $\pio$.   There is an isomorphism of schemes
 \begin{equation} a\mapsto a .\tp: A \overset{\sim}{\To} \pio \ . 
 \end{equation}    
 Via this identification, the action of $A$ on $\pio$ can be computed explicitly. In particular, its antisymmetrization is  given by the Ihara bracket. The dual coaction
  \begin{equation} \label{Gonchcoact}
\mathcal{O}(\pio) \overset{\Delta}{\To} \mathcal{O}(\pio) \otimes_{\Q} \mathcal{O}(\pio) \cong \mathcal{O}(A) \otimes_{\Q} \mathcal{O}(\pio)
  \end{equation} 
  was computed  by Goncharov in \cite{GG}, Theorem 1.2, except that the two right-hand factors are interchanged. The formula involves $\mathcal{O}( {}_a \Pi_b)$ for
  all $a,b\in \{0,1\}$ but can easily be rewritten in terms of  $ \mathcal{O}(\pio)$ only (this is the content of Properties $\Io$, $\Ii$, $\Iiii$   below). It follows that the coaction 
  $(\ref{Ocoaction})$ is obtained  by composing $\Delta$ of $(\ref{Gonchcoact})$ with the map
  \begin{eqnarray} \label{acton1}
  \mathcal{O}(\pio) & \To&  \Ao \\
\phi & \mapsto &  g \mapsto \phi(g. \tp) \ , \nonumber 
  \end{eqnarray}
applied to the left-hand factor of  $  \mathcal{O}(\pio) \otimes_{\Q}   \mathcal{O}(\pio)$. Note that since $\G_m$ acts trivially on $\tp$, the map $(\ref{acton1})$ necessarily loses information about the weight grading.
\subsection{Definition of motivic MZVs} Let $I \subset \mathcal{O}(\pio)$ be the kernel of the map $\dch$ $(\ref{dchdual})$. It describes the $\Q$-linear relations between multiple zeta values. Let 
$\Jmt \subseteq I$ be the largest graded ideal contained in $I$ which is closed under the coaction $(\ref{Ocoaction})$. 
\begin{defn} Define the graded coalgebra of motivic multiple zeta values to be 
\begin{equation} \label{Hodef} \Ho = \mathcal{O}(\pio) / \Jmt \ .\end{equation}
A word $w$ in the letters $0$ and $1$ defines   an element in  $(\ref{opi01})$. Denote its image in $\Ho$ by
\begin{equation} \label{defnImot} \Imot(0;w;1) \in \Ho\ ,
\end{equation}
 which we shall call a motivic iterated integral.  
For $n_0\geq0 $ and $n_1,\ldots, n_r \geq 1$, let 
  \begin{equation} \label{zetamotdefn} 
  \zetam_{n_0}(n_1,\ldots, n_r) = \Imot(0; \underbrace{0,\ldots,0}_{n_0},\underbrace{1,0,\ldots,0}_{n_1}, \ldots,\underbrace{1,0,\ldots,0}_{n_r}  ;1) \ .
  \end{equation}
In the case when $n_0=0$, we shall  simply write this $\zetam(n_1,\ldots, n_r)$
and call it a motivic multiple zeta value. As usual, we denote the grading on $\Ho$ by a subscript.
 \end{defn}
The coaction  $(\ref{Ocoaction})$ induces a coaction  
$\Delta:\Ho \rightarrow \Ao \otimes_{\Q} \Ho $.  By the discussion at the end of \S\ref{sectprelim}, it can be computed from $(\ref{Gonchcoact})$, i.e., the following diagram commutes:
\begin{equation} \label{coprodistherightone}
 \begin{array}{ccc}
 \mathcal{O}(\pio) & \To    & \Ao \otimes_{\Q} \mathcal{O}(\pio)   \\
 \downarrow &   & \downarrow   \\
  \Ho & \To  & \Ao \otimes_{\Q} \Ho   \ .
\end{array}
\end{equation}
Furthermore, the map $\dch$ $(\ref{dchdual})$ factors through $\Ho$. The resulting homomorphism  from $\Ho$ to $\R$ shall be called  the period map, which we denote by
\begin{equation} \label{perdef} per: \Ho \To  \R\ . \end{equation}

\begin{rem} The ideal $\Jmt$ could be  called  the ideal of motivic relations. A linear combination  $R\subset \mathcal{O}(\pio)$ of words is a relation between  motivic multiple zetas  if 

1). $R$ holds numerically (i.e.,  $per (R)=0$), 

2). $R'$ holds numerically for all  transforms $R'$ of $R$ under the coaction $(\ref{Ocoaction})$.

\noindent 
This argument is used in $\S 4$ to lift certain relations from multiple zetas to their motivic versions, and can be made into a kind of numerical algorithm   (see \cite{Br}). 
\end{rem}

\subsection{The role of $\zetam(2)$}  It is also convenient to consider the dual point of view. Let $\ZZ=\mathrm{Spec} \, \Ho$. It is the Zariski closure of the $\GMT$-orbit of $\dch$, 
 i.e., 
$$\ZZ = \overline{\GMT .dch }\ .$$
Thus $\ZZ$ is a subscheme of the extension of scalars  $\pio\otimes_{\Q}\R$, but is in fact defined over $\Q$ since $dch$ is Betti-rational.  
Let $\tau$ denote the action of $\G_m$ on $\pio$. The map $\tau(\lambda)$  multiplies elements of degree $d$ by $\lambda^d$.
 Let us choose a rational point  $\gamma \in \ZZ(\Q)$ which is  even, i.e., $\tau(-1)\gamma=\gamma$ (see \cite{DG}, \S5.20). 
 Since    $\GPU$ is the quotient of $\GU$ through which it acts on $\ZZ$, we obtain an isomorphism 
\begin{eqnarray} \label{gammamap}
\GPU \times \A^1 &\overset{\sim}{\To} & \ZZ  \\ 
(g,t) & \mapsto & g \, \tau({\sqrt{t}}) . \gamma  \nonumber 
\end{eqnarray}
The  parameter $t$  is retrieved  by taking the coefficient of $e_0 e_1$ in the series $g \, \tau({\sqrt{t}}).\gamma$. Thus $(\ref{gammamap})$ gives rise to  an  isomorphism of graded algebra comodules
\begin{eqnarray} \label{Hotens}
\Ho \cong \Ao \otimes_{\Q}\Q [\zetam(2)] \ ,
\end{eqnarray}
which depends on $\gamma$, where $\Delta(\zetam(2))= 1 \otimes \zetam(2)$. Most of our constructions will not in fact depend on this choice of $\gamma$, but we may fix it once  and for all. 
Since the leading term of $dch$ is $\tp$, we have  $\lim_{t \rightarrow 0} \tau(t) \dch = \tp$, which shows that 
$$ \GPU.\tp  \subseteq \ZZ\ . $$
Thus the map  $\Ho \rightarrow \Ao$ is induced by 
$(\ref{acton1})$, and sends $\zetam(2)$ to zero.

\begin{defn}Let us denote the graded ring of affine functions on $\GU$ over $\Q$ by
$$ \Amt = \mathcal{O}(\GU)\ , $$
and   define $\Homt=\Amt\otimes_{\Q}\Q[f_2],$
 where  the elements $f_2^k$  are in degree $2k$. Thus $\Homt$ is a graded algebra comodule over $\Amt$, and its  grading shall  be called
 the weight hereafter.  We shall also denote the coaction by 
 \begin{equation} \label{homtdelta}
 \Delta: \Homt \To \Amt \otimes_{\Q} \Homt\ .
 \end{equation} 
 It is uniquely determined by the property $\Delta f_2 = 1\otimes f_2$.
 \end{defn}
In conclusion,  the inclusion $\Ao\rightarrow \Amt$ which is dual to the quotient map $\GU \rightarrow \GPU$, induces 
an injective morphism of graded algebra comodules
\begin{eqnarray} \label{HotoHomt}
\Ho \To \Homt
\end{eqnarray}
which sends $\zetam(2)$ to $f_2$, by  $(\ref{Hotens})$. The map $(\ref{HotoHomt})$ implicitly depends on the choice of $\gamma$.
Abusively, we shall sometimes 
identify  motivic multiple zeta values with their images under $(\ref{HotoHomt})$, i.e., 
as elements in $\Homt$.

\subsection{Main properties} In order to write down the formula for the coaction on the motivic iterated integrals, we must slightly extend the notation    $(\ref{defnImot})$. 
 For all sequences $a_0,\ldots, a_{n+1} \in \{0,1\}$,  there are   elements 
 \begin{equation} \label{Sec2Imotdef}
 \Imot(a_0;a_1,\ldots,a_n ;a_{n+1}) \in \Ho_n \end{equation} 
 given  by $(\ref{defnImot})$ if $a_0=0$ and $a_{n+1}=1$, and uniquely defined by the   properties below.
  
  \2skip
 \begin{description}
  \item[I0]   If $n\geq 1$,   $\Imot(a_0;a_1,\ldots,a_n; a_{n+1})=0 \hbox{ if } a_0 =a_{n+1}$  or $a_1=\ldots =a_n$\ .  \2skip
  \item[I1]   $\Imot(a_0;a_1)=1$ for all $a_0,a_1\in \{0,1\}$  and $\zetam(2) = f_2$.  \2skip
  \item[I2] \emph{Shuffle product (special case).} We have  for $k\geq0, n_1,\ldots, n_r\geq 1$, 
  $$ \zetam_k(n_1,\ldots, n_r) = (-1)^k   \sum_{i_1+ \ldots +i_r=k} \! \binom{n_1+i_1-1}{i_1}\ldots \binom{n_r+i_r-1}{i_r} \zetam(n_1+i_1,  \ldots, n_r+i_r) \nonumber$$ \2skip 
  \item[I3] \emph{Reflection formulae}. For all $a_1,\ldots, a_n\in \{0,1\}$,  $$\Imot(0;a_1,\ldots, a_n;1) = (-1)^n \Imot(1;a_n,\ldots, a_1;0) =  \Imot(0;1-a_n,\ldots, 1-a_1;1)$$ \2skip
\end{description}
\noindent
The motivic multiple zeta values   of \cite{GG} are, up to a possible  sign,  the images of $(\ref{Sec2Imotdef})$ under the map 
$\pi:\Ho \rightarrow \Ao$ which sends $\zetam(2)$ to zero. We have shown ($(\ref{Gonchcoact})$,  $(\ref{coprodistherightone})$):

\begin{thm} \label{thmGonchCoproduct}  The coaction for the motivic multiple zeta values
\begin{equation}\label{fulldelta} \Delta: \Ho \rightarrow \Ao \otimes_{\Q} \Ho\ \end{equation}
is given by the same formula as \cite{GG}, Theorem 1.2, with the factors interchanged. In particular, 
if $a_0,\ldots, a_{n+1}\in \{0,1\}$, then $\Delta\, \Imot (a_0;a_1,\ldots, a_{n};a_{n+1})$ equals
\begin{equation}\label{defGcoproduct}  \sum_{i_0<i_1<  \ldots<i_{k+1} \atop i_0=0, i_{k+1}=n+1}  
  \pi \Big( \prod_{p=0}^k  \Imot (a_{i_p}; a_{{i_p}+1}, \tdots, a_{i_{p+1}-1} ;a_{i_{p+1}}) \Big)\otimes  \Imot (a_0;a_{i_1},\tdots, a_{i_k}; a_{n+1}) \end{equation}
where  the first  sum is also  over all values of $k$ for   $0\leq k\leq n$.
\end{thm}

Lastly, the period map $per: \Ho \rightarrow \R$ can be computed as follows:
$$ per\big( \Imot(a_0;a_1,\ldots, a_n;a_{n+1} ) \big)=  \int_{a_0}^{a_{n+1}} \omega_{a_1}\ldots \omega_{a_n} \ ,\nonumber $$
where $\omega_{0} = {dt \over t}, \omega_1 = {dt \over 1-t}$, and  the right-hand side is a   shuffle-regularized  iterated integral
 (e.g., \cite{DG} \S5.16).  Note that the sign of $\omega_1$  varies in the literature.
 Here,  the signs are chosen such that the  period of the motivic multiple zeta values are 
\begin{equation} \label{periodmapsigns}
per  \big( \zetam(n_1,\ldots, n_r)\big)  =  \zeta(n_1,\ldots, n_r) \quad \hbox{ when } n_r\geq 2\ .
\end{equation}

\subsection{Structure of $\Homt$}
The structure of $\MT(\Z)$ is determined  by the data:
$$\mathrm{Ext}_{\MT(\Z)}^1(\Q(0),\Q(n)) \cong \left\{
                           \begin{array}{ll}
                             \Q\  & \hbox{if } n\geq 3 \hbox{ is odd}\ ,  \nonumber \\
                             0\   & \hbox{otherwise} \ ,
                           \end{array}
                         \right.
$$
 and the fact that the $\mathrm{Ext}^2$'s vanish.  Thus $\MT(\Z)$ is equivalent to the category of representations of a group  scheme $\GMT$ over $\Q$, which is a semi-direct product
 $$ \GMT\cong  \GU \rtimes \mathbb{G}_m \ ,$$
 where $\GU$ is a prounipotent group whose Lie algebra $\mathrm{Lie}\, \GU$ is isomorphic to the free Lie algebra with one generator $\sigma_{2i+1}$ in every degree 
 $-2i-1$, for $i\geq 1$.
Let  $\Amt$ be  the graded  ring of functions on $\GU$,  where the  grading is  with respect to the action of $\G_m$. 
Since the degrees of the $\sigma_{2i+1}$ tend to minus infinity, no information is   lost in passing to  the graded  version  $(\mathrm{Lie}\, \GU)^{gr}$ 
of  $ \mathrm{Lie}\, \GU$ (Proposition 2.2 $(ii)$ of \cite{DG}). 
By the above, $\Amt$  is  non-canonically isomorphic to the graded dual of the universal envelopping algebra of $(\mathrm{Lie}\, \GU)^{gr}$ over $\Q$. 
 We shall denote this by 
 $$\U' = \Q\langle f_3,f_5,\ldots \rangle\ .$$
Its underlying vector space has  a basis consisting of   non-commutative words in symbols $f_{2i+1}$ in degree $2i+1$,  and the multiplication is given by the shuffle product $\sha$.  The coproduct
 is given by  the following  
 deconcatenation formula: 
  \begin{eqnarray} \label{Udelta}
 \Delta:\U' & \rightarrow  & \U' \otimes_{\Q} \U'  \\
\Delta(f_{i_1}\ldots f_{i_n}) & =  &  \sum_{k=0}^{n} f_{i_1}\ldots f_{i_{k}}\otimes  f_{i_{k+1}}\ldots f_{i_n}  \ , \nonumber
\end{eqnarray}
when $n\geq 0$. Let us consider the following universal comodule
\begin{equation}
\U=  \Q\langle f_3,f_5,\ldots \rangle \otimes_{\Q}  \Q[f_2]
\end{equation}
where $f_2$  is of degree $2$,  commutes with all generators $f_{2n+1}$ of odd degree, and the coaction $\Delta: \U \rightarrow \U' \otimes_{\Q} \U$ satisfies $\Delta(f_2) =1 \otimes f_2$.  The degree will also be referred to as the weight.
By the above discussion,   there exists  a non-canonical isomorphism 
\begin{equation} \label{HisomU}
\phi:\Homt\cong \U 
\end{equation} 
 of algebra comodules which sends $f_2$ to $f_2$.   These notations are useful for explicit  computations, for which it can be  convenient to vary the choice of map $\phi$  (see \cite{Br}).

 \begin{lem} Let $d_k=\dim \U_k$, where $\U_k$ is the graded piece of $\U$ of weight $k$.  Then
\begin{equation} \label{enumeration} 
\sum_{k\geq 0} d_k t^k = {1\over  1-t^2-t^3}\ . \end{equation}
In particular, $d_0=1, d_1=0, d_2=1$ and  $d_k= d_{k-2}+d_{k-3}$ for $k\geq 3$.
\end{lem} 
\begin{proof} The Poincar\'e series of $\Q\langle f_3,f_5, \ldots \rangle$ is given by 
${1\over 1-t^3-t^5-\ldots } = {1-t^2 \over 1-t^2-t^3}$.
If we multiply by the Poincar\'e series ${1\over 1-t^2}$  for $\Q[f_2]$, we obtain $(\ref{enumeration})$.
\end{proof}
 \begin{defn} \label{defnFNnotation} It will be convenient to define an element $f_{2n} \in \U_{2n}$, for $n\geq 2$, by  
 $$f_{2n}=b_n f_2^n \ , $$where $b_n \in \Q^{\times}$ is the constant in Euler's relation $\zeta(2n)=b_n \zeta(2)^n$. 
\end{defn}
 
 Let us denote the Lie coalgebra of indecomposable elements of $\U'$ by
 \begin{equation} \label{Lnotcurlydefn}
 L= {\U'_{>0}\over \U'_{>0} \U'_{>0} }\ ,
 \end{equation}
and  for any $N\geq 1$, let $\pi_{N}: \U'_{>0}\rightarrow  L_N$ denote the quotient map followed by projection onto the graded part of weight $N$.
For any $r\geq 1$,    consider the map
   \begin{equation} \label{UDr} D_{2r+1}: \U  \To L_{2r+1} \otimes_{\Q} \U \end{equation}
 defined by composing $\Delta'= \Delta - 1\otimes id$ with $\pi_{2r+1}\otimes id$. Let
 \begin{equation} \label{UDltN} D_{<N}=\bigoplus_{1 < 2i+1< N} D_{2i+1} \ . 
 \end{equation}
 
 \begin{lem} \label{lemkerDonU} $(\ker D_{<N}) \cap \U_N = \Q\, f_N\ .$
 \end{lem}
 \begin{proof} Every element $\xi \in \U_N$ can be uniquely written in the form
 $$\xi = \sum_{1< 2r+1 < N} f_{2r+1} v_r + c f_N$$
 where $c\in \Q$,  $v_r \in \U_{N-2r-1}$, and the multiplication on the right-hand side is the concatenation product. The graded dual of $L$ is isomorphic to the free Lie algebra with generators $f^{\vee}_{2r+1}$ in degrees $-2r-1$ dual to the $f_{2r+1}$.  Each element $f^{\vee}_{2r+1}$
 defines a map $f^{\vee}_{2r+1}: L \rightarrow \Q$ which sends $f_{2r+1}$ to $1$.
  By definition, 
 $$(f^{\vee}_{2r+1} \otimes id )\circ D_{2r+1} \xi = v_r\ ,$$
 for all $1<2r+1<N$. It follows immediately that if $\xi$ is in the kernel of $D_{<N}$ then it is of the form $\xi = cf_N$. Since $D_{<N} f_N=0$, the result follows.
 \end{proof}
\section{Cogenerators of the coalgebra} 

\subsection{Infinitesimal coaction}
In order to simplify the formula $(\ref{defGcoproduct})$, let 
\begin{equation}\label{Ldef}
\Lo = {\Ao_{>0} \over \Ao_{>0} \Ao_{>0}}\end{equation}
denote the Lie coalgebra of $\Ao=\Ho/ f_2\Ho$, and let $\pi:\Ho_{>0}\rightarrow \Lo$ denote the quotient map.  Since $\Lo$ inherits a weight grading from $\Ho$, let  $\Lo_{N}$ denote the elements of $\Lo$ of homogeneous weight $N$, and let $p_N: \Lo \rightarrow \Lo_{N}$ be  the projection map.

\begin{defn} \label{defnDronH} By analogy with $(\ref{UDr})$ and $(\ref{UDltN})$, define for every $r\geq 1$   a map
$$D_{2r+1} : \Ho \To \Lo_{2r+1}\otimes_{\Q} \Ho $$
 to be  $( \pi\otimes id) \circ \Delta'$, where $\Delta'=\Delta-1\otimes id$,  followed by  $p_{2r+1}\otimes id$.  Let
 \begin{equation}\label{DNdef} D_{<N}= \bigoplus_{3\leq 2r+1<N} D_{2r+1}\ . \end{equation}
 \end{defn}
  
It follows from  this definition  that the maps $D_{n}$, where $n=2r+1$, are derivations:
\begin{equation} \label{Drarederivations} 
D_{n} (\xi_1 \xi_2) = (1 \otimes  \xi_1) D_{n} (\xi_2) + (1 \otimes \xi_2) D_{n}(\xi_1)  \quad \hbox{ for all } \xi_1,\xi_2 \in \Ho \ .
\end{equation} 
By  $(\ref{defGcoproduct})$, the action of $D_n$  on $ \Imot(a_0;a_1,\ldots, a_N;a_{N+1})$ is given    by:
\begin{equation}\label{mainformula}   
 \sum_{p=0}^{N-n} \pi \big(\Imot(a_{p} ;a_{p+1},\tdots, a_{p+n}; a_{p+n+1})\big) \otimes   \Imot(a_{0}; a_1, \tdots, a_{p}, a_{p+n+1}, \tdots ,  a_N ;a_{N+1})  \ .\end{equation}
  We call the  sequence of consecutive elements $(a_p; a_{p+1},\ldots,a_{p+n}; a_{p+n+1})$ on the left  a \emph{subsequence of length} $n$  of the original sequence, 
  and $( a_0;a_1, \tdots, a_{p}, a_{p+n+1}, \tdots ,  a_N;a_{N+1})$  will be called the \emph{quotient sequence}, by analogy with the Connes-Kreimer coproduct.

\subsection{Zeta elements and kernel of $D_{<N}$} 
\begin{lem} \label{lemzetaelements} Let $n\geq 1$. The zeta element   $\zetam(2n+1)\in \Ho$ is non-zero and satisfies
$$
\Delta\, \zetam(2n+1) = 1 \otimes \zetam(2n+1) + \pi( \zetam(2n+1)) \otimes 1\ . 
$$
Furthermore, Euler's relation $\zeta(2n)= b_n \zeta(2)^n$, where $b_n \in \Q$, holds for the $\zetam$'s:
\begin{equation} \label{zetam2n} \zetam(2n) = b_n \zetam(2)^n \ .
\end{equation}
 \end{lem}
\begin{proof} Consider $\zetam(N) = \Imot(0;10^{N-1};1)$. By relation $\Io$,  its strict subsequences of length at least one are killed by $\Imot$, and  so
$\Delta \, \zetam(N) = 1\otimes \zetam(N) + \pi(\zetam(N)) \otimes 1$ by  
$(\ref{mainformula})$.
From the structure of $\U$, it  follows that an isomorphism  $\phi$  $(\ref{HisomU})$ maps $\zetam(2n+1)$ to $\alpha_n f_{2n+1}$, for some $\alpha_n \in \Q$, and $\zetam(2n)$ to $\beta_n f^n_2$ for some $\beta_n\in \Q$.
  Taking the period map yields $\alpha_n\neq 0$, and $\beta_n=\zeta(2n) \zeta(2)^{-n}=b_n$.
\end{proof}

We can therefore normalize our choice of  isomorphism  $(\ref{HisomU})$  so that
$$\Ho \subseteq  \Homt \overset{\phi}{\To} \U  $$
maps   $\zetam(2n+1)$ to $f_{2n+1}$. By definition \ref{defnFNnotation},     we can therefore write:
\begin{equation} \label{phinormed}
\phi(\zetam(N))=f_N \quad   \hbox{ for all } \quad  N\geq 2 \ . 
\end{equation}
In particular, $\zetam(2)$ and $\zetam(2n+1)$, for $n\geq 1$, are algebraically independent in $\Ho$.

\begin{thm} \label{propkerD} Let $N\geq 2$.  The kernel of $D_{<N}$ is one-dimensional in  weight $N$:  
$$\ker D_{<N} \cap \Ho_N = \Q\,\zetam(N)\ .$$
\end{thm} 
\begin{proof} This  follows from lemma \ref{lemkerDonU}, via such an isomorphism $\phi$. 
 \end{proof}
 Note that the map $\Lo\rightarrow L$ of Lie coalgebras induced by the inclusion $\Ao \subseteq\Amt\cong \U'$ is also injective, by standard results  on  Hopf algebras. 
\subsection{Some relations between motivic multiple zeta values}  Using theorem $\ref{propkerD}$ we lift relations between multiple zeta values to their 
motivic versions. Hereafter let 
$2^{\{n\}}$ denote a sequence of $n$ consecutive 2's, and let $\zetam(2^{\{0\}})=1\in \Ho$. For any word $w$ in the alphabet $\{2,3\}$, define the weight of $w$ to be $2 \deg_2 w + 3 \deg_3 w$.
\begin{lem} \label{lemlev0}  The element  $\zetam(2^{\{n\}})$ is a rational multiple of $\zetam(2)^n$. 
\end{lem}
\begin{proof} For reasons of parity, every  strict subsequence of $(0;1010\ldots 10;1)$ of odd length begins and ends in the same symbol, and   corresponds to a zero motivic iterated integral by  $\Io$. Therefore
$D_{2r+1} \zetam(2^{\{n\}})=0$ for all $r\geq 1$. By proposition \ref{propkerD} it  is  a multiple of $\zetam(2n)$, that is $\zetam(2)^n$.
The multiple  is equal to $\zeta(2^{\{n\}})/\zeta(2)^n>0$.
\end{proof} 
The coefficient in the lemma can be  determined by the well-known  formula  
\begin{equation} \label{zeta222}
\zeta(2^{\{n\}}) = { \pi^{2n} \over  (2n+1)!}\ .
\end{equation}
 We need the following trivial observation, valid for $n\geq 1$:
\begin{equation} \label{12nshuffle} 
\zetam_1(2^{\{n\}}) =  -2 \sum_{i=0}^{n-1} \zetam(2^{\{i\}}32^{\{n-1-i\}})\ , 
\end{equation} 
which  follows immediately from  relation $\Iii$.

\begin{lem} \label{coprodlevel1} Let $a,b\geq 0$ and $1\leq r \leq  a+b$.  Then   
$$D_{2r+1}\,  \zetam(2^{\{a\}}32^{\{b\}}) = \pi(\xi^r_{a,b}) \otimes \zetam(2^{\{a+b+1-r\}})\ ,$$
 where $\xi^r_{a,b}$ is given by (sum over all indices $\alpha,\beta\geq0 $ satisfying $\alpha+\beta+1=r$): 
 $$\xi^r_{a,b}= \sum_{{\alpha\leq a \atop \beta \leq b}} \zetam(2^{\{\alpha\}} 3 2^{\{\beta\}})  - \sum_{\alpha \leq a \atop \beta< b} \zetam(2^{\{\beta\}} 3 2^{\{\alpha\}}) +\big(\mathbb{I}(b\geq r) -\mathbb{I}(a\geq r)\big)\, \zetam_1(2^{\{r\}}) \ ,$$
 The symbol $\mathbb{I}$ denotes the indicator function.
\end{lem}
\begin{proof}  The element $\zetam(w)$, where $w=2^{\{a\}}32^{\{b\}}$,  is represented by the sequence
$$\Imot(0;10\ldots 1{\oo}10 \ldots 10;1)\ .$$
By parity, every subsequence of length $2r+1$ which does not straddle the  subsequence ${\oo}$    begins and ends in the same symbol. Its motivic iterated integral vanishes by $\Io$,
so it does not contribute to $D_{2r+1}$.
 The remaining  subsequences are of the form 
 $$\Imot(0;\underbrace{10\ldots 10}_{\alpha}1\oo\underbrace{10 \ldots 10}_{\beta};1)=\zetam(2^{\{\alpha\}}32^{\{\beta\}})\ ,$$ 
 where $\alpha \leq a$, $\beta\leq b$, and $\alpha+\beta+1=r$, which gives rise to the first sum,
 or
 $$\Imot(1;\underbrace{01\ldots 01}_{\alpha}\oo1\underbrace{01 \ldots 01}_{\beta};0)=-\zetam(2^{\{\beta\}}32^{\{\alpha\}})\ ,$$ which gives  the second sum, by $\Iiii$. 
 In this case $\beta\neq b$. Finally, we can also have 
$$\Imot(1;\underbrace{01\ldots 01}_{r}\os;\os)=-\zetam_1(2^{\{r\}})\ , \hbox{ or } \Imot(\os; \os \underbrace{10\ldots 10}_{r};1)=\zetam_1(2^{\{r\}})\ ,$$
which gives rise to the last two terms. The quotient sequences are the same in all cases,  and equal to $\Imot(0;10\ldots 10;1)$. This  proves the formula. \end{proof}

The following trivial observation follows from lemma \ref{lemzetaelements}:
\begin{equation} \label{trivob}
D_{2r+1} \zetam(N) = \pi(\zetam(2r+1)) \otimes  \delta_{N,2r+1}  \qquad N\geq 2, r\geq 1 \ ,  
\end{equation} 
where $\delta_{i,j}$ denotes the Kronecker delta.

\begin{cor} \label{corlevel1} Let $w$ be a word in $\{2,3\}^{\times}$ of weight $2n+1$ which has many 2's and a single 3.  Then
there exist unique numbers $\alpha_i \in \Q$ such that
\begin{equation}\label{leveloneshape} 
 \zetam(w) = \sum_{i=1}^{n} \alpha_i\,  \zetam(2i+1) \, \zetam(2^{\{n-i\}}) \ . 
\end{equation}
\end{cor}
\begin{proof} By induction on the weight of $w$. Suppose the result is true for $1\leq n< N$. Then for  a word $w=2^{\{a\}}32^{\{b\}}$ of weight $2N+1$, the elements $\xi^r_{a,b}$ of the previous lemma, for $1\leq r<N$,  are of the form $(\ref{leveloneshape})$. In particular, there exists some rational number $\alpha_r$ such that  $\xi^r_{a,b} \equiv \alpha_r\,  \zetam(2r+1)$ modulo products.
 It follows that for $1\leq r < N$, 
 $$D_{2r+1}\, \zetam(w) = \alpha_r \pi(\zetam(2r+1)) \otimes \zetam(2^{\{N-r\}})$$
and so the left and right-hand sides of
  $(\ref{leveloneshape})$ have the same image under $D_{<2N+1}$ by $(\ref{Drarederivations})$ and $(\ref{trivob})$.
By theorem \ref{propkerD} they differ by  a rational multiple of $\zetam(2N+1)$.
\end{proof}

The previous corollary leads us to the following definition.
\begin{defn} Let $w$ be a word in $\{2,3\}^{\times}$ of weight $2n+1$, with a single $3$. Define the \emph{coefficient} $c_w \in \Q$   of $\zetam(w)$
to be the coefficient of $\zetam(2n+1)$ in equation $(\ref{leveloneshape})$.
By $(\ref{12nshuffle})$, we can define  the coefficient $c_{12^n}$ of  $\zetam_1(2^{\{n\}})$ in the same way. It satisfies
\begin{equation} \label{c12id} c_{12^n}=-2 \sum_{i=0}^{n-1}  c_{2^{i}32^{n-i-1}}  \ .   \end{equation}
Thus for any $w \in \{2, 3\}^{\times}$ of weight $2n+1$ which contains a single 3,  we have:
\begin{equation} \label{pizetaw}
\pi(\zetam(w)) = c_{w} \pi(\zetam(2n+1) )  \ . 
\end{equation}
\end{defn}

Later on, we shall work not only with the actual coefficients $c_w\in \Q$, whose properties are described in the following section, but also with 
formal versions $C_w$ of these coefficients  (later to be specialized to $c_w$) purely  to simplify calculations.

\begin{lem} \label{lem010} For all $n\geq 1$,   the coefficient $c_{12^n}$ is equal to  $2(-1)^n$.  We have
$$
 \zetam_1(2^{\{n\}}) = 2 \sum_{i=1}^{n}  (-1)^i \zetam(2i+1) \zetam(2^{\{n-i\}})\ .
$$
\end{lem}
\begin{proof}
Granting the motivic stuffle product formula  \cite{Rac,Soud} for our version of  motivic multiple zeta values  in which $\zetam(2)$ is non zero, we have:
 \begin{eqnarray}
 \zetam(3)\zetam(2^{\{n-1\}})  &=  &\sum_{i=0}^{n-1} \zetam(2^{\{i\}}32^{\{n-1-i\}}) + \sum_{i=0}^{n-2} \zetam(2^{\{i\}} 5 2^{\{n-2-i\}}) \nonumber \\
 \zetam(5)\zetam(2^{\{n-2\}})  &=  &\sum_{i=0}^{n-2} \zetam(2^{\{i\}}52^{\{n-2-i\}}) + \sum_{i=0}^{n-3} \zetam(2^{\{i\}} 7 2^{\{n-3-i\}}) \nonumber \\
 \vdots \qquad \qquad & \vdots & \qquad \qquad \qquad\qquad \qquad \vdots \nonumber \\
 \zetam(2n-1)\zetam(2) & = & \zetam(2n-1,2)+\zetam(2,2n-1) + \zetam(2n+1) \ .\nonumber
 \end{eqnarray} 
  Taking the alternating sum of each row  gives the equation: 
 $$\sum_{i=0}^{n-1} \zetam(2^{\{i\}}32^{\{n-1-i\}}) = -2 \sum_{i=1}^{n}  (-1)^i \zetam(2i+1) \zetam(2^{\{n-i\}})  \ .$$
This implies the lemma, by $(\ref{12nshuffle})$.
  Alternatively,  we can   use the general  method for lifting relations from real multiple zeta values to their motivic versions given in  \cite{Br}. For this, it  only suffices to consider
 the above relations modulo products and modulo $\zetam(2)$ to obtain the coefficients of $\zetam(2n+1)$, which leads to  the same result.
 \end{proof}

\section{Arithmetic of the coefficients $c_w$} \label{sect4}
The key arithmetic input  is an evaluation of certain multiple zeta values  which is due to Don Zagier.
Using the operators $D_{<N}$ we shall lift this result to motivic multiple zetas.
 First, let us define for any $a,b,r\in \N$, 
$$A^r_{a,b} =  \binom{2r}{2a+2} \  \hbox{ and  } \ B^r_{a,b} =\bigl(1-2^{-2r}\bigr)\binom{2r}{2b+1}\ . $$
Note that $A^r_{a,b}$ (respectively $B^r_{a,b}$) only depends on $r,a$ (resp. $r,b$).

\begin{thm} \label{thmZagierthm} (Don Zagier \cite{Zagier}).  Let $a,b\geq 0$. Then 
$$\zeta(2^{\{a\}}32^{\{b\}})=  2\,\sum_{r=1}^{a+b+1}(-1)^r (A^r_{a,b}-B^r_{a,b})\, \zeta(2r+1)\, \zeta(2^{\{a+b+1-r\}})\,\;. $$
\end{thm}

 \subsection{Motivic version}  To prove that the arithmetic formula lifts to the level of motivic multiple zeta values  requires showing that it is compatible with the coaction $(\ref{fulldelta})$. This  
is equivalent to the following compatibility condition between the coefficients.
 
 \begin{lem}\label{lemcollapse}
For any $a,b\geq0$,  and $1\leq r\leq a+b+1$ we have
\begin{eqnarray} 
 \sum_{{\alpha\leq a \atop \beta \leq b}}  A^{r}_{\alpha,\beta}   - \sum_{\alpha \leq a \atop \beta< b} A^{r}_{\beta,\alpha} +    \mathbb{I}(b\geq r) -  \mathbb{I}(a\geq r)  & = & A^{r}_{a,b} \nonumber \\
  \sum_{{\alpha\leq a \atop \beta \leq b}}  B^{r}_{\alpha,\beta}   - \sum_{\alpha \leq a \atop \beta< b} B^{r}_{\beta,\alpha}   & = & B^{r}_{a,b} \nonumber
\end{eqnarray}
where all sums are over sets of indices $\alpha,\beta\geq0$ satisfying $\alpha+\beta+1=r$.
\end{lem}
\begin{proof} Exercise, using  
 $A^{\alpha+\beta+1}_{\alpha,\beta} = A^{\alpha+\beta+1}_{\beta-1,\alpha+1}$   and $B^{\alpha+\beta+1}_{\alpha,\beta}=B^{\alpha+\beta+1}_{\beta,\alpha}$.
 \end{proof}
We now show that Zagier's theorem lifts to  motivic multiple zetas.
\begin{thm} \label{thmMotZagier} Let $a,b\geq 0$. Then 
\begin{equation} \label{zl1motform} \zetam(2^{\{a\}}32^{\{b\}})=  2\,\sum_{r=1}^{a+b+1}(-1)^r (A^r_{a,b}-B^r_{a,b})\,  \zetam(2r+1)\, \zetam(2^{\{a+b+1-r\}})\;. 
\end{equation}
In particular, if $w=2^{\{a\}}32^{\{b\}}$, then the coefficient $c_w$ is given by
\begin{equation}\label{cwformula}
c_w =   2\, (-1)^{a+b+1} \big( A^{a+b+1}_{a,b}-B^{a+b+1}_{a,b}\big)\;. 
\end{equation}
\end{thm}

\begin{proof} The proof is by induction on the weight. Suppose that $(\ref{zl1motform})$ holds for all $a+b <N$ and let $a,b\geq 0$ such that $a+b=N$. Then  by lemma \ref{coprodlevel1},
$$ D_{2r+1} \, \zetam(2^{\{a\}}32^{\{b\}}) =  \pi(\xi^r_{a,b})  \otimes \zetam(2^{\{a+b+1-r\}}) \ ,$$
for $1\leq r\leq N$. 
The second formula in lemma \ref{coprodlevel1} and $(\ref{pizetaw})$ implies that 
$$\pi(\xi^r_{a,b})  =  \Big( \sum_{{\alpha\leq a \atop \beta \leq b}} c_{2^\alpha 32^\beta}  - \sum_{\alpha \leq a \atop \beta< b}c_{2^{\beta}32^{\alpha}} +  c_{12^r} \mathbb{I}(b\geq r) -  c_{12^r} \mathbb{I}(a\geq r)  \Big)  \pi(\zetam(2r+1))\ , $$
where   $1\leq r\leq N $ and   the sum is over $\alpha, \beta\geq 0$ satisfying  $\alpha+ \beta+1=r$. 
By induction hypothesis, and
the fact that $c_{12^r}=2(-1)^r$,  the term in brackets is 
$$ 
2 (-1)^r\Big[ \sum_{{\alpha\leq a \atop \beta \leq b}} \Big(A^{r}_{\alpha,\beta} -B^{r}_{\alpha,\beta}\Big)  - \sum_{\alpha \leq a \atop \beta< b}\Big(A^{r}_{\beta,\alpha}-B^{r}_{\beta,\alpha}\Big) +   \mathbb{I}(b\geq r) -  \mathbb{I}(a\geq r) \Big] \ ,$$
where all $\alpha+\beta+1=r$. By lemma \ref{lemcollapse}, this  collapses to
 $$ 2 (-1)^r \big(  A^{r}_{a,b}- B^{r}_{a,b}  \big)\ .  $$
 Putting the previous expressions together, we have shown that 
 $$ D_{2r+1} \, \zetam(2^{\{a\}}32^{\{b\}}) =    2(-1)^r \big(  A^{r}_{a,b}- B^{r}_{a,b}  \big)   \pi(\zetam(2r+1)) \otimes \zetam(2^{\{a+b+1-r\}})\ . $$
It follows by $(\ref{Drarederivations})$ and  $(\ref{trivob})$ that the difference 
$$\Theta=\zetam(2^{\{a\}}32^{\{b\}})-  2\, \sum_{r=1}^{a+b+1} (-1)^r (A^r_{a,b}-B^r_{a,b})\, \zetam(2r+1)\,  \zetam(2^{\{a+b+1-r\}})  $$
satisfies $D_{2r+1}\Theta=0$ for all $r\leq a+b$.  By theorem \ref{propkerD}, there is an $\alpha\in \Q$ such that
$$\Theta = \alpha \, \zetam(2a+2b+3)\ .$$
Taking the period map implies  an analogous relation for the ordinary multiple zeta values. By theorem \ref{thmZagierthm}, the constant $\alpha$ is $2(-1)^{a+b+1}(A^{a+b+1}_{a,b}-B^{a+b+1}_{a,b})$, which completes the induction step, and hence the proof of the theorem.
\end{proof}
\subsection{$2$-adic properties of $c_w$} The coefficients $c_w$ satisfy some arithmetic properties which are crucial for the sequel, and follow  immediately from theorem  \ref{thmMotZagier}. 
 
 Let $p$ be a prime number. Recall that for any rational number $x\in \Q^{\times}$, its  $p$-adic valuation $v_p(x)$ is the integer $n$ such that
$x=p^n {a \over b}$, where $a,b$ are relatively prime to $p$. We set $v_p(0)=\infty$. For any word $w\in \{2,3\}^\times$, let $\widetilde{w}$ denote the same word but written in reverse order.

\begin{cor} \label{cor2adic}  Let $w$ be any word of the form $2^a32^b$ of weight $2a+2b+3 = 2n+1$.  
It is obvious from formula $(\ref{cwformula})$  that $c_w \in \Z[{1\over 2}]$. 
Furthermore, the 
  $c_w$ satisfy:
  
  \begin{enumerate}
  \item  $ c_{w} - c_{\widetilde{w}} \in 2 \, \Z$\ ,
  \item $ v_2(c_{32^{a+b}}) \leq v_2(c_{w})\leq 0 $\ .
\end{enumerate}
 \end{cor}
 \begin{proof}  Property $(1)$ follows from the symmetry $B^{a+b+1}_{a,b}=B^{a+b+1}_{b,a}$. Indeed,
 $$c_{2^a32^b}-c_{2^b32^a} =  \pm 2 (A^{a+b+1}_{a,b}-A^{a+b+1}_{b,a})\in 2\,  \Z\ .$$
Let $n=a+b+1$ be fixed. Clearly,  $v_2(c_{2^a32^b}) = v_2 ( 2\times 2^{-2n}\times  \binom{2n}{2b+1})$, and so
 $$v_2(c_{2^a32^b}) = 1-2n+ v_2\big(\textstyle{\binom{2n}{2b+1}}\big)\ .$$
 Writing $\binom{2n}{2b+1} = {2n \over 2b+1} \binom{2n-1}{2b}$, we obtain
 $$v_2(c_{2^a32^b}) = 2-2n+v_2(n) + v_2\big( \textstyle{\binom{2n-1}{2b}}\big) $$
which is  $\leq 0$, and furthermore, $v_2\big( \textstyle{\binom{2n-1}{2b}}\big)$ is minimal when $b=0$.  This proves $(2)$.
 \end{proof}

\section{The level filtration and $\partial_{N,\ell}$ operators}

\begin{defn}
Let $\Ho^{2,3}\subset \Ho$ denote the $\Q$-subspace  spanned by
\begin{equation}\label{Ho23gen}
\zetam(w) \quad  \hbox{ where } \quad  w\in \{2,3\}^{\times}\ .
\end{equation}
It inherits the weight grading from $\Ho$. The weight of  $\zetam(w)$ is $2\deg_2 w+ 3\deg_3 w$.
\end{defn}

\begin{defn}Consider the unique map \begin{equation} \rho: \{2,3\}^{\times} \To \{0,1\}^{\times}
\end{equation} 
 such that $\rho(2)=10$ and $\rho(3)=100$, which respects the concatenation product.
The motivic iterated integral which corresponds to $\zetam(w)$ is $\Imot(0;\rho(w);1)$.
\end{defn}

\begin{lem} The coaction $(\ref{fulldelta})$   gives a map
$$\Delta: \Ho^{2,3} \To \mathcal{A} \otimes_{\Q} \Ho^{2,3}\ .$$
\end{lem}
\begin{proof}
This follows from $(\ref{defGcoproduct})$ together with  the fact that  $\Imot$ vanishes on sequences which begin and end in the same symbol ($\Io$). Thus, for  $w\in \{2,3\}^{\times}$,
every  quotient sequence of $(0;\rho(w);1)$ which occurs non-trivially on the right-hand side of  the coaction   is  again of the form $(0;w';1)$,  where $w'$ is   a word in $10$ and $100$.
\end{proof}

\subsection{ The level filtration}

\begin{defn} Let $w\in \{2,3\}^{\times}$ be a word in the letters $2$ and $3$. We define  the  \emph{level} of $w$ to be $\deg_3 w$, the number of occurrences of the letter `3' in $w$. 
Denote the induced increasing filtration on $\Ho^{2,3}$ by  $F_{\bullet}$, i.e., $F_{\ell}\Ho^{2,3}$ is the $\Q$-vector space spanned by
 $$ \{ \zetam(w): w\in \{2,3\}^\times \hbox{ such that }  \deg_3 w\leq \ell\} \ .$$ 
\end{defn}

The level  counts the number of occurrences  of the sequence  `$00$' in $\rho(w)$.
Alternatively, it is  given by the weight minus twice the depth (number of $1$'s) in $\rho(w)$. Thus the level filtration   takes even (resp. odd) values in even (resp. odd) weights. 
 The level filtration  is motivic in the following sense:
$$\Delta \big( F_{\ell} \Ho^{2,3} \big)  \subseteq  \mathcal{A}\otimes_{\Q} F_{\ell}\Ho^{2,3} \ .$$
This is because any sequence of $0$'s and $1$'s contains  at least as many   `$00$'s  than any of   its quotient sequences.
 It follows that  the maps $D_{2r+1} :\Ho^{2,3} \rightarrow \Lo_{2r+1}\otimes_{\Q} \Ho^{2,3}$, where $\Lo$ is the Lie coalgebra of $\Ao=\Ho/f_2\Ho$, also preserve the level filtration. In fact, more is true.
 \begin{lem}  \label{lemdroplevel} For all $r\geq 1,$ 
$$  D_{2r+1} \big( F_\ell \Ho^{2,3} \big) \subseteq \Lo_{2r+1}\otimes_{\Q} F_{\ell-1} \Ho^{2,3} \ .$$
\end{lem}
\begin{proof} Let $w\in \{2,3\}^\times$ of level $\ell$, so $\rho(w)$ contains exactly $\ell$ sequences $00$.   If a subsequence of odd length of $(0;\rho(w);1)$
begins and ends in the same symbol it is killed by $\Imot$ by $\Io$ and does not contribute to $D_{2r+1}$. Otherwise,  it must necessarily contain at least one $00$, and so the associated quotient sequence is  of strictly smaller level. 
\end{proof}
Thus for all $r,\ell \geq 1$,  we obtain a map
\begin{equation} \label{grFin5pt1}
\gr^F_{\ell} D_{2r+1}: \gr^F_{\ell} \Ho^{2,3} \To \Lo_{2r+1} \otimes_{\Q} \gr^F_{\ell-1} \Ho^{2,3}\ .
\end{equation}

\subsection{The maps $\partial_{N,\ell}$} In order to simplify  notations, let us define $\zeta_{2r+1}\in \Lo_{2r+1}$  by 
 $$\zeta_{2r+1}=\pi(\zetam(2r+1))  \ ,\quad \hbox{ where } r\geq 1\ .$$

\begin{lem} \label{lemdelta}   Let $r, \ell \geq 1$. Then  the map $(\ref{grFin5pt1})$ satisfies 
$$\gr^F_{\ell} D_{2r+1} \big( \gr^F_{\ell} \Ho^{2,3} \big) \subseteq \Q\, \zeta_{2r+1} \otimes_{\Q} \gr^F_{\ell-1}\Ho^{2,3}\ .$$
\end{lem}
\begin{proof} Let $w\in \{2,3\}^{\times}$ of level $\ell$, and let $\Imot(0;\rho(w);1)$ be the corresponding motivic iterated integral.  From the definition of  $D_{2r+1}$, we  have
\begin{equation}\label{formofdelta}
\gr^F_{\ell} D_{2r+1} (\zetam(w)) = \sum_{\gamma}  \pi(\Imot(\gamma))\otimes \zetam(w_\gamma)\ ,
\end{equation}
where  the sum is over all subsequences $\gamma$ of $(0;\rho(w);1)$ of length $2r+1$, 
and $w_\gamma$ is the corresponding  quotient sequence. If $\gamma$ contains more than one subsequence $00$ then $w_{\gamma}$ is of level $< \ell-1$ and so does not contribute.
If $\gamma$ begins and ends in the same symbol, then $\Imot(\gamma)$ is zero.  One checks that  $\Imot(\gamma)$ can be  of four  remaining  types:
\begin{enumerate}
\item  $\Imot(0;10\ldots 1\oo10 \ldots 10;1) = \zetam(2^{\{\alpha\}}32^{\{\beta\}})$
\item  $\Imot(1;01\ldots 01\oo1 \ldots 01;0) = -\zetam(2^{\{\alpha\}}32^{\{\beta\}})$
\item $\Imot(\os;\os1\ldots 10;1)  = \zetam_1(2^{\{r\}})$
\item $ \Imot(1;01\ldots 1\os;\os) = - \zetam_1(2^{\{r\}})$
\end{enumerate}
By corollary \ref{corlevel1} and $(\ref{12nshuffle})$, in every case we have $\pi(\Imot(\gamma)) \in \Q\, \zeta_{2r+1}\subset \Lo_{2r+1}$. 
The coefficient of $\zeta_{2r+1}$ in   $\pi(\Imot(\gamma))$ is either   $\pm c_{2^\alpha32^\beta}$ or $\pm c_{12^r}$.
\end{proof}
Sending $\zeta_{2r+1}$ to $1$ gives  a canonical identification of 1-dimensional vector spaces:
\begin{equation}\label{Qid}
\Q\, \zeta_{2r+1}  \overset{\sim}{\To}  \Q\ .  \end{equation}

\begin{defn}  For all $N, \ell\geq 1$, let $\partial_{N,\ell}$ be the linear  map 
\begin{eqnarray}
\partial_{N,\ell}:  \gr^F_\ell \Ho^{2,3}_N & \To & \bigoplus_{  1< 2 r+1  \leq N  } \gr^F_{\ell-1} \Ho^{2,3}_{N-2r-1} \qquad \Big( = \gr^F_{\ell-1} \Ho^{2,3}_{<N-1}\Big) \nonumber  
\end{eqnarray}
defined by  first applying 
$$\bigoplus_{1 < 2r+1 \leq N}\gr^F_{\ell} D_{2r+1}\Big|_{\gr^F_\ell \Ho^{2,3}_N}
$$
and then sending all $\zeta_{2r+1}$ to $1$ via  $(\ref{Qid})$. Note that since $F_0 \Ho_0^{2,3}= \Ho_0^{2,3}$ the case $2r+1=N$ only plays a role when $\ell=1$.
\end{defn}
Our goal is to show that the  maps  $\partial_{N,\ell}$ are  injective for $\ell \geq 1$. 

\subsection{Matrix representation for $\partial_{N,\ell}$}

 \begin{defn} 
Let $\ell \geq 1$, and let  $B_{N,\ell}$  denote the set of words $w\in \{2,3\}^{\times}$ of weight $N$ and level $\ell$,   in reverse lexicographic order  for the ordering $3<2$. Let $B'_{N,\ell}$ denote the set
of words $w \in \{2,3\}^{\times}$ of all positive weights $<N-1$ (including the empty word) and level $\ell-1$, also in reverse  lexicographic order.
\end{defn}
If we write $N=2m+3\ell$, and if $\ell\geq 1$, then clearly
$$|B_{N,\ell}| = \binom{{m+\ell}}{\ell}  = \sum_{0\leq m'\leq m} \binom{m'+\ell-1}{\ell-1} = |B'_{N,\ell}|.$$
Define a set of words
\begin{equation} \label{Sdef} S= \{w: w\in \{2,3\}^{\times} \hbox{ of level } 1\} \cup \{ 12^k: k\in \N \}\  .\end{equation}
  Let $\ell\geq 1$, and let $w\in B_{N,\ell}$. By $(\ref{formofdelta})$ and the proof of lemma \ref{lemdelta}, 
  \begin{equation} \label{deltaf}
  \partial_{N,\ell} \,\zetam(w) = \sum_{w'\in B'_{N,\ell}} f^w_{w'} \zetam(w')\ ,\end{equation}
  where $f^w_{w'}$ is a $\Z$-linear combination of numbers $c_{w''}\in \Q$, for $w''\in S$.
   \begin{defn} For $\ell\geq 1$, let $M_{N,\ell}$  be the matrix  $(f^w_{w'})_{w\in B_{N,\ell}, w'\in B'_{N,\ell}}$, where $w$ corresponds to the rows, and $w'$ the columns.
 \end{defn}
Note that we have not yet proved that the elements $\zetam(w)$ for $w\in B_{N,\ell}$ or $B'_{N,\ell}$  are linearly independent. The matrix $M_{N,\ell}$ represents the map $\partial_{N,\ell}$  nonetheless.

\subsection{Formal coefficients} It is convenient to consider a formal version of the map $(\ref{deltaf})$ in which the coefficients $f^w_{w'}$ are replaced by symbols. 
Each matrix element $f^w_{w'}$ of $M_{N,\ell}$ is a linear combination of $c_w$, where $w\in S$. 
Therefore, let $\Z^{(S)}$ denote the free $\Z$-module generated by symbols $C_w$, where $w\in S$, and formally define a map 
$$\partial^f_{N,\ell} : \gr^F_{\ell} \Ho_N^{2,3}  \To  \Z^{(S)} \otimes_{\Z} \gr^F_{\ell-1}\Ho^{2,3}_{<N-1} $$
from the formula $(\ref{formofdelta})$ by replacing the coefficient $\pm c_w$ of $\pi(\Imot(\gamma))$ with its formal representative   $\pm C_{w}$.
 Likewise, let $M^f_{N,\ell}$ be the matrix with coefficients in $\Z^{(S)}$ which is the formal  version of  $M_{N,\ell}$. There is a linear map 
\begin{eqnarray} \label{mudef}
\mu:\Z^{(S)} & \To  & \Q   \\
C_w & \mapsto  & c_w\ . \nonumber
\end{eqnarray}
Then by definition  $\partial_{N,\ell} = (\mu\otimes id)  \circ \partial^f_{N,\ell}$ and $M_{N,\ell}$ is obtained from the matrix $M^f_{N,\ell}$ by applying $\mu$ to  all of its entries. 
\subsection{Example}
In weight 10 and level 2, the matrix $M^f_{10,2}$ is as  follows. The words in the first row (resp. column) are the elements of $B'_{10,2}$ (resp. $B_{10,2}$) 
in order.

\begin{table}[h]
\begin{center}
{\small \begin{tabular}{|c|c|c|c|c|c|c|}
\hline
$ $ &223 &   232& 23 & 322 & 32 & 3 \\ 
\hline 
2233 &$C_{3} \! - \!C_{12}$& &$C_{23} \! - \!C_{32} \! - \!C_{122}$ & & & $C_{223} \! - \!C_{ 322}$  \\
2323 & &$C_{3} \! - \!C_{12}$& $C_{23}$ & & & \\
2332 &                   & & $C_{32}$& $C_{23} \! - \!C_{32}$ & & \\
3223 & $C_{12}$  &  & $C_{32}\!-\!C_{23}\!+\!C_{122}$ & $C_3 \! - \!C_{12}$ & $C_{23} \! - \!C_{122}$  & $C_{322}$ \\
3232 &                   & $C_{12}$  &         &                               & $C_{32}$                                           &     $C_{232}$          \\
3322 &                   &                   &         & $C_{12}$              &$C_{32} \! - \!C_{ 23} \! + \!C_{122}$  &     $C_{322}$      \\
\hline 
\end{tabular}}
\end{center}
\end{table}%

All blank entries are zero. Let us check the entry for 
$$\zetam(3,3,2,2)=\Imot(\underset{a_0}{0};\underset{a_1}{1}00100101\underset{a_{10}}{0};\underset{a_{11}}{1})$$
Number the elements of the right-hand sequence $a_0,\ldots, a_{11}$ for reference, as shown.
The non-vanishing terms in  $\gr^F_2 D_3$ correspond to the subsequences commencing at $a_0,a_1,a_3,a_4,a_5$. They all give rise to the same quotient sequence, and we get:
$$ {\small \pi(\Imot(0;100;1)) \otimes \Imot(0;1001010;1) + \pi(\Imot(1;001;0))\otimes \Imot(0;1001010;1)+} $$
$$\pi(\Imot(0;100;1)) \otimes \Imot(0;1001010;1) + \pi(\Imot(1;001;0))\otimes \Imot(0;1001010;1) + $$
$$\pi(\Imot(0;010;1)) \otimes \Imot(0;1001010;1)$$
which gives $(C_3-C_3 +C_3-C_3+C_{12})\, \zetam(3,2,2)=C_{12}\, \zetam(3,2,2)$.  The non-zero terms in $\gr^F_2 D_5$ correspond to subsequences commencing at 
$a_3,a_4,a_5$. 
They give
$$\pi(\Imot(0;10010;1)) \otimes \Imot(0;10010;1) + \pi(\Imot(1;00101;0))\otimes \Imot(0;10010;1) $$
$$+ \pi(\Imot(0;01010;1))\otimes \Imot(0;10010;1)$$
which is $(C_{32}-C_{23}+C_{122} ) \,\zetam(3,2)$. Finally, the only non-zero term  in $\gr^F_2 D_7$ corresponds to the subsequence commencing at 
$a_3$, giving
$$ \small{\pi(\Imot(0;1001010;1)) \otimes \Imot(0;100;1)} $$
 which is $C_{322}\,  \zetam(3)$. The matrix $M_{10,2}$ is obtained by replacing each $C_w$ by $c_w$.

\section{Calculation  of  $\partial_{N,\ell}$}
Let  $I \subseteq \Z^{(S)}$ be the submodule spanned by elements of the form:
\begin{equation}\label{Idef}
C_w-C_{\widetilde{w}} \quad  \hbox{ for }   w\in {\{2,3\} }^{\times} \hbox{ of level } 1 \ , \quad \hbox{ and } \quad  C_{12^k}   \hbox{ for }  k\in \N\ ,
\end{equation}
where $\widetilde{w}$ denotes the reversal of the word $w$.
 We show that   modulo $I$ the maps  $\partial^f_{N,\ell}$ 
act by deconcatenation.
\begin{thm}  \label{thmmodI} Let $w$ be a word in  $\{2,3\}^{\times}$ of weight $N$, level $\ell$.  Then 
$$\partial^f_{N,\ell} \,  \zetam(w) \equiv \sum_{w = uv \atop \deg_3 v=1} C_v \, \zetam(u) \pmod I $$ 
where the sum is over all deconcatenations $w=uv$, where $u,v\in \{2,3\}^{\times}$,  and where $v$ is of level 1, i.e., contains exactly one letter `3'.
\end{thm}

\begin{proof} A subsequence of length $2r+1$ of the element $\Imot(0;\rho(w);1)$  which corresponds to $\zetam(w)$ can be of  the following types (compare the proof of lemma \ref{lemdelta}):
\begin{enumerate}
\item The subsequence is an alternating sequence of $1$'s and $0$'s,  and does not contain any consecutive $0$'s.  For reasons of parity, its first and last elements are equal. Thus  by $\Io$, this case does not contribute.
\item The subsequence contains one set of consecutive $0$'s and  is  of the form $(0;w';1)$, where the inital $0$ and final $1$ are directly above  the two arrows below: 
\vspace{-0.03in}
$$\cdots 01\underset{\uparrow}{\os}\os10 \cdots1010\underset{\uparrow}{1}0\cdots\ .$$
The word $w'$ consists of the  $r+1$ symbols  strictly inbetween the arrows. This subsequence corresponds to $\Imot(\os;\os10\ldots 10;1)$ which is $\zetam_1(2^{\{r\}})$. The contribution is therefore
$C_{12^r} \in I$.  With similar notations, the case
\vspace{-0.03in}
$$ \cdots0\underset{\uparrow}{1} 0101 \cdots 01\os \underset{\uparrow}{\os}10\cdots$$
corresponds to $\Imot(1;01\ldots 01\os;\os)$ and contributes $-C_{12^r} \in I$, by relation $\Iiii$.
\item The subsequence is of one of two  forms:
\vspace{-0.03in}
$$\underset{a}{0}\underset{b}{1}010\cdots 1\oo \cdots 10\underset{a'}{1}\underset{b'}{0}$$
and contains exactly one set of consecutive $0$'s. The subsequence from $a$ to $a'$ is of the form 
$\Imot(0;1010\ldots 1\oo 10\ldots 10;1)$ and contributes a $\zetam(2^{\{\alpha\}} 32^{\{\beta\}})$. The subsequence from $b$ to $b'$ is of the form
$\Imot(1;010\ldots 1\oo 10\ldots 101;0)$. Using  relation  $\Iiii$,  this contributes a $(-1)^{2r+1}\zetam(2^{\{\beta\}} 32^{\{\alpha\}})$.
Thus the total contribution of these two subsequences is 
$$C_{2^\alpha 32^\beta } - C_{2^\beta 32^\alpha } \in I\ .$$
\item  The subsequence has at least two sets of consecutive $0$'s. This case does not contribute, since if the subsequence has level $\geq 2$,  the quotient sequence
has level $\leq \ell-2$ which maps to zero in $\gr^F_{\ell-1} \Ho^{2,3}$.
\end{enumerate}
Thus every non-trivial subsequence is either of the form  $(2)$, or pairs up with its immediate neighbour to the right to give a contribution of type $(3)$.
The only remaining possibility is the final subsequence of $2r+1$ elements, which has no immediate right neighbour. This gives rise  to a single contribution  of the form
$$C_v \zetam(u)$$
where $w=uv$ and $v$ has weight $2r+1$. If $v$ has level strictly greater than $1$ then by the same reasoning as $(4)$ it does not contribute. This proves the theorem.
\end{proof}

\begin{cor}\label{coruppertri}
 The matrices $M^f_{N,\ell}$  modulo  $I$    are upper-triangular. Every entry which lies on  the diagonal is of the form $C_{32^{r-1}}$, 
 where $r\geq 1$, and  
 every entry lying above it in the same column is of the form $C_{2^a3 2^b}$, where $a+b+1=r$. 
\end{cor}
 
 \begin{proof}  Let $\ell \geq 1$ and 
consider the map 
\begin{eqnarray} 
B'_{N,\ell-1} & \To &  B_{N,\ell} \\
u & \mapsto & u\, 32^{\{r-1 \}} \nonumber
\end{eqnarray}
where $r\geq 1$ is the unique integer such that the weight of $ u\, 32^{\{r-1\}}$ is equal to $N$.  This map is a bijection and preserves the ordering of both $ B'_{N,\ell-1}$ 
and $ B_{N,\ell}$. It follows from theorem \ref{thmmodI}
that the diagonal entries of $M^f_{N,\ell}$  modulo  $I$ are of the form $C_{32^{r-1}}$. Now let $u \in B'_{N,\ell-1}$.   All non-zero entries      of $M^f_{N,\ell}$  modulo  $I$   in the column indexed by $u$ lie in the rows indexed by $u 2^{\{a\}} 3 2^{\{ b\}}$, where $a+b=r-1$. Since
$$u 2^{\{a\}} 3 2^{\{ b\}} < u \, 32^{\{r-1 \}}\ ,$$
it is  upper-triangular, and the entry in row $u 2^{\{a\}} 3 2^{\{ b\}}$ and column $u$ is $C_{2^a3 2^b}$.
 \end{proof}

\section{Proof of the main theorem}

\subsection{$p$-adic lemma}
We need  the following elementary lemma.

\begin{lem} \label{padiclemma} Let $p$ be a prime, and let $v_p$ denote the $p$-adic valuation. Let $A=(a_{ij})$ be a square $n\times n$ matrix with rational coefficients such that
\begin{eqnarray} \label{padiccond}
&(i)&  v_p(a_{ij}) \geq 1 \quad \hbox{ for all } i>j \\
&(ii)&   v_p(a_{jj})=  \underset{i}{\min}\{v_p(a_{ij})\}\leq 0  \quad \hbox{ for all }  j \ . \nonumber
\end{eqnarray}
Then $A$ is invertible.
\end{lem} 

\begin{proof} We show that the determinant of $A$ is non-zero. The determinant is an alternating sum of products of elements of $A$, taken one from each row and column.
Any such monomial $m\neq 0$ has $k$ terms  on or  above  the diagonal, in columns $j_1,\ldots, j_k$, and $n-k$ terms strictly below the diagonal.
By $(i)$ and $(ii)$ its $p$-adic valuation is 
$$
v_p(m)  \geq   (n-k) + \sum_{r=1}^k v_p(a_{j_rj_r}) $$
If $m$ is not the monomial $m_0$ in which all terms lie on the diagonal, then $k<n$, and 
$$v_p(m) >   \sum_{i=1}^n v_p(a_{ii}) =v_p(m_0)\ . $$
It follows that $v_p(\det(A)) =  v_p(m_0)=\sum_{i=1}^n v_p(a_{ii}) \leq 0 <\infty$, so $\det(A) \neq 0$.
\end{proof}

\begin{rem} Another way to prove this lemma is simply to scale each column of the matrix $A$ by a suitable power of $p$ so that $(i)$ remains true, and so that the diagonal elements have
valuation exactly equal to $0$. The new  matrix has $p$-integral coefficients by $(ii)$, is upper-triangular mod $p$, and  is invertible on the diagonal.
\end{rem}

\begin{thm} \label{thmlevel1} For all $N, \ell\geq 1$, the matrices $M_{N,\ell}$ are invertible.
\end{thm}

\begin{proof} We show that $M_{N,\ell}$ satisfies the conditions $(\ref{padiccond})$ for $p=2$. The entries of $M_{N,\ell}$ are deduced from those of $M^f_{N,\ell}$ by applying the map $\mu$ of $(\ref{mudef})$. 
It follows from corollary \ref{cor2adic} (1) and lemma \ref{lem010}   that  the generators $(\ref{Idef})$ of $I$ map to even integers under $\mu$ i.e.,
$$\mu(I) \subset 2\, \Z\  .$$
By corollary  $\ref{coruppertri}$, this implies that $M_{N,\ell}$ satisfies $(i)$. Property $(ii)$ follows from  corollary \ref{cor2adic} $(2)$, since the diagonal entries
of $M^f_{N,\ell} \mod I$ are   $C_{32^k}$, $k\in\N$. 
\end{proof}

\subsection{Proof of the main theorem} \label{sectfinalproof}

\begin{thm}  The elements
$\{\zetam(w): w\in \{2,3\}^{\times}\}$
are linearly independent.
\end{thm} 
\begin{proof}  By induction on the level. The elements of level zero are of the form 
$\zetam(2^{\{n\}})$ for $n\geq 0$, which by lemma \ref{lemlev0} are linearly independent. Now suppose that 
$$\{\zetam(w): w\in \{2,3\}^\times \ , \quad w \hbox{ of level } = \ell\}$$
are independent. Since the  weight is a grading on the motivic multiple zeta values,  we can assume that any non-trivial linear relation between $\zetam(w)$, for $w\in \{2,3\}^\times$ of level $\ell+1$ is of  homogeneous weight $N$.  By theorem \ref{thmlevel1}, the map $\partial_{N,\ell+1}$ is injective and therefore  gives  a non-trivial  relation  of  strictly smaller  level, a contradiction. Thus $\{\zetam(w): w\in \{2,3\}^\times \hbox{ of level }\ell+1\}$ are linearly independent, which completes the induction step.   The fact that
the operator $D_{\leq N}$  strictly decreases the level 
 (lemma \ref{lemdroplevel}), and that its level-graded pieces    $\partial_{N,\ell}$ are injective   implies  that there can be no non-trivial relations between elements $\zetam(w)$ of different levels.
\end{proof}
It follows that 
\begin{equation} \label{lowerbound} \dim \Ho_N^{2,3}  =  \#\{w\in \{2,3\}^\times  \hbox{ of weight } N\} =  d_N\ ,
\end{equation}
where $d_N$ is the dimension of $\U_N$ $(\ref{enumeration})$. The inclusions
$$\Ho^{2,3}\subseteq \Ho \subseteq \Homt$$
are therefore all equalities, since their dimensions in graded weight $N$ are  equal.
The equality $\Ho^{2,3}=\Ho$ implies the following corollary:
\begin{cor} Every motivic multiple zeta value $\zetam(a_1,\ldots, a_n)$  is a $\Q$-linear combination of $\zetam(w)$,  for $w\in\{2,3\}^\times$.
\end{cor}
 
This  implies conjecture \ref{conj2}. The equality $\Ho=\Homt$ implies  conjecture \ref{conj1}.

\section{Remarks}
By adapting the above argument, one could show that the elements
$$\{\zetam(w): \hbox{ where } w \hbox{ is a Lyndon word in } \{2,3\}^\times\}$$
are algebraically independent in $\Ho$. That their periods should be algebraically independent was conjectured in  \cite{DataMine}.  It might also be interesting to  try to use the matrices $M_{N,\ell}$ to compute the multiplication law on the basis $(\ref{Hbasis})$.

The geometric meaning of theorem \ref{thmMotZagier} is not  clear.   The term $(1-2^{-2r}) \zeta(2r+1)$ which comes from  $B^{a,b}_r$  
can  be interpreted  as an alternating sum  and suggests that the formula should be viewed   as an identity between motivic iterated integrals on $\Pro^1\backslash \{0,\pm 1,\infty\}$.  It would be interesting to find a direct motivic proof of theorem \ref{thmMotZagier}  along these lines. Apart from the final step of \S\ref{sectfinalproof}, 
the only other place where we use the structure of the category $\MT(\Z)$ is in  theorem \ref{thmMotZagier}. A proof of theorem  \ref{thmMotZagier} using standard relations would
give a purely combinatorial proof that $\dim \Ho_N \geq d_N$.

The  argument in this paper could  be dualized to take place in Ihara's  algebra, using his pre-Lie  operator  (\cite{DG}, $(5.13.5)$) instead of Goncharov's formula  $(\ref{defGcoproduct})$. This  sheds some light on  the appearance of the deconcatenation coproduct in theorem \ref{thmmodI}.

Finally, we should point out that the existence of a different explicit basis for multiple zeta values was  apparently announced many years ago by J.   Ecalle.

\section{Acknowledgements}
Very many thanks to Pierre Cartier and Pierre Deligne for a thorough reading of this text and many  corrections and comments.  
 This work was  supported by  European Research Council grant no.  257638: `Periods in algebraic geometry and physics'. 

\bibliographystyle{plain}
\bibliography{main}

\end{document}